\documentclass[amscd, amssymb, verbatim, 12pt]{amsart}
\usepackage{amssymb}

\def\be{\begin{eqnarray}}
\def\ee{\end{eqnarray}}
\def\bea{\begin{eqnarray*}}
\def\eea{\end{eqnarray*}}

\def\e{\epsilon}

\def\tr{\Delta}
\def\a{\alpha}

\def\b{\beta}

\def\la{\langle}
\def\ra{\rangle}

\def\a{\alpha}

\def\vp{\varphi}

\def\b{\beta}
\def\na{\nabla}

\def\O{\Omega}
\def\vp{\varphi}

\newtheorem{defi}{Definition}[section]
\newtheorem{exam}[defi]{Example}
\newtheorem{lem}[defi]{Lemma}
\newtheorem{remark}[defi]{Remark}
\newtheorem{thm}[defi]{Theorem}

\newtheorem{cor}[defi]{Corollary}

\numberwithin{equation}{section}

\begin{document}
\title[Variational Characterizations of the Total Scalar Curvature]
{Variational Characterizations of the Total Scalar Curvature and Eigenvalues of the Laplacian}

\author{Seungsu Hwang}
\address{Department of Mathematics,
Chung-Ang University, 221 HeukSuk-dong, DongJak-ku, Seoul, Korea 156-756}
\email{seungsu@cau.ac.kr}

\author{Jeongwook Chang}
\address{Department of Mathematics Education,
Dankook University, 126 Jukjeon-dong, Yong-in, Gyeong-gi, Korea 448-701}
\email{jchang@dankook.ac.kr}

\author{Gabjin Yun$^*$}

\thanks{* Corresponding author.}

\thanks{The first author was supported by the Basic Science Research Program through the National Research Foundation
of Korea(NRF) funded by the Ministry of Education, Science and Technology(Grant No. 2011-0005211),
the second author by the Ministry of Education, Science and Technology(Grant No. 2011-0005235),
and the third author by the Ministry of Education, Science and Technology(Grant No. 2011-0007465).}

\address{Department of Mathematics\\
Myong Ji University\\
San 38-2, Nam-dong, Yong-in\\
Gyeong-gi, Korea, 449-728}
\email{gabjin@mju.ac.kr}



\maketitle
\begin{abstract}
For the dual operator $s_g'^*$ of the linearization $s_g'$ of the scalar curvature function,
it is well-known that if $\ker s_g'^*\neq 0$, then $s_g$ is a non-negative constant.
In particular, if the Ricci curvature is not flat, then
$ {s_g}/(n-1)$ is an eigenvalue of the Laplacian of the metric $g$.
In this work, some variational characterizations were performed for the space $\ker s_g'^*$.
To accomplish this task, we introduce a fourth-order elliptic differential
operator $\mathcal A$ and a related geometric invariant $\nu$.
We prove that $\nu$ vanishes if and  only if $\ker s_g'^* \ne 0$, and
if the first eigenvalue of the Laplace operator is large compared to its scalar curvature,
then $\nu$ is positive and $\ker s_g'^*= 0$.
Furthermore, we calculated the lower bound on $\nu$ in the case of $\ker s_g'^* =  0$.
We also show that if there exists a function which is $\mathcal A$-superharmonic and
the Ricci curvature has a lower bound, then the first non-zero eigenvalue of the Laplace operator has an upper bound.
\end{abstract}

\setlength{\baselineskip}{20pt}

\section{Introduction}

Let $M$ be a compact smooth $n$-manifold (without a boundary).
The space of all Riemannian metrics, $\mathcal M$, on $M$ is then open in the
space of symmetric $2$-tensors, ${\mathcal S}^2(M)$, for the
compact-open topology or the $W^{k,p}$-topology, where $W^{k,p}$ denotes the Sobolev space.
For a Riemannian metric $g$ and a symmetric $2$-tensor $h$, the
differential of the scalar curvature at $g$ in the direction $h$,
$s_g'(h)$, is given by
\be
s_g'(h) = -\Delta_g {\rm tr}(h) + \delta_g(\delta_g h) - g(r_g, h),\label{eqn2009-9-13-2}
\ee
where $\Delta_g$ is the negative Laplacian of $g$, and $r_g$  and
$\delta_g$ denote the Ricci curvature and divergence operator of $g$, respectively (\cite{besse}).
In addition, the $L^2$-adjoint operator $s_g'^*$ of $s_g'$ is given by
\be
s_g'^*(f) = Ddf - (\Delta_g f)g - f r_g,\label{eqn2009-9-12-3}
\ee
where $Ddf$ denotes the Hessian of $f$ with respect to the metric $g$. Note that both $s_g'$ and
$s_g'^*$ are linear second order differential operators.

In this paper, we consider a fourth-order elliptic differential operator
$\mathcal A = s_g'\circ s_g'^* : C^\infty(M) \to C^\infty(M)$.
The existence of homogeneous or non-homogeneous solutions to $\mathcal A$ is deeply related to the kernel space of $s_g'^*$.
For example, in \cite{bour} and \cite{fi-ma}, Bourguignon, and Fischer-Marsden proved, respectively,
 that if $\ker s_g'^* \ne 0$, then either  $(M, g)$ is
Ricci-flat and $\ker s_g'^* = {\Bbb R}\cdot 1$, or the scalar curvature is a strictly positive constant and
$\frac{s_g}{n-1}$ is an eigenvalue of the Laplacian. On the other hand, if $\ker s_g'^* = 0$,
then for any function $\eta \in C^\infty(M)$ there exists
a unique function $u \in C^\infty(M)$ such that $\mathcal A(u) = \eta$ (Theorem~\ref{thm2011-9-6-2}).
In fact, the condition $\ker s_g'^* = 0$ implies
the injectivity of $s_g'^*$ and the surjectivity of $s_g'$.
In order to perform variational characterizations of the condition $\ker  s_g'^* \ne 0$,
we introduce  a geometric invariant $\nu$ which is defined by
$$
\nu = \inf \left\{\int_M \vp \mathcal A\vp\, dv_g\right\},
$$
where the infimum is taken over all functions $\vp \in H^2(M) = W^{2,2}(M)$ with $\int_M \vp^2 = 1$.
Here $H^2(M) = W^{2,2}(M)$ denotes the Sobolev space which is $L^2$ up to the second (weak) derivatives.

A basic result related to the invariant $\nu$ is the following.

\vspace{.1in}
\noindent
{\bf Theorem A}\,\,
$\nu >0$ if and  only if $\ker  s_g'^* = 0$.
\vspace{.1in}

For the case $\ker  s_g'^* =  0$, we give a lower bound on $\nu$ and its relationship with the first non-zero eigenvalue of the Laplacian.
We also show that if the first eigenvalue is large compared to the scalar curvature,
then $\nu$ is positive and  $\ker  s_g'^*= 0$.
In addition, if $M$ is the product of two standard spheres of the same dimension,
then $\nu$ is exactly equal to the dimension of the sphere.

\vspace{.1in}
\noindent
{\bf Theorem B}\,\,
Let $M = S^m\times S^m$ $(m \ge 2)$ with the standard product metric. Then
$$
\nu = m = \frac{\dim(M)}{2}.
$$
\vspace{.1in}

 We also obtain upper bounds for the first non-zero eigenvalue of the Laplace operator
 when $\mathcal A$ satisfies a condition on evaluating functions.
 We say that  a Riemannian manifold $(M, g)$ satisfies the $\mathcal A$-superharmonic condition if there
exists a smooth function $\vp$ such that $M^+_\vp \ne \emptyset$ and $\mathcal A \vp \le 0$ on $M^+_\vp$,
and  $\Delta \vp = 0$ on the boundary $\partial M_\vp^+$ of $M_\vp^+$. For example, if $M$ is the standard sphere, then
the first eigenfunction of the Laplacian satisfies those conditions. In general, any compact Riemannian manifold
$(M, g)$  with a positive scalar curvature and $\ker s_g'^* \ne 0$ satisfies the $\mathcal A$-superharmonic condition.

One of our main results is the following.

\vspace{.1in}
\noindent
{\bf Theorem C}\,\,
Let $(M^n, g)$ be a compact $n$-dimensional Riemannian manifold with a positive constant scalar curvature $s_g$.
Suppose that $(M, g)$ satisfies the $\mathcal A$-superharmonic condition.
If $r_g \ge k$ for $0 < k \le 2\left(1 - \sqrt{1- \frac{1}{n}}\right)s_g$, then
 the first eigenvalue of the Laplacian satisfies
\be
\lambda = \lambda_1(M) \le \frac{2s_g-k + \sqrt{k^2 - 4ks_g + \frac{4s_g^2}{n}}}{2(n-1)}.\label{eqn2011-9-6-1}
\ee
\vspace{.1in}

\noindent
The inequality (\ref{eqn2011-9-6-1}) is sharp since the equality holds for the standard sphere.
In performing analysis with the operator $\mathcal A$, main difficulty is that we cannot apply the theory of second order elliptic PDE
directly since $\mathcal A$ is a fourth-order differential operator.
In fact, there are not much known results for fourth-order ellitptic PDE.

The kernel space of $s_g'^*$ plays an important role in the critical point equation arising from
the total scalar curvature functional. Let ${\mathcal M}_1$ be the set of all smooth Riemannian metrics of unit volume on $M$,
and let $\mathcal C \subset {\mathcal M}_1$ be the set of all smooth Riemannian metrics on $M$
with a constant scalar curvature, i.e.,
$$
\mathcal C = \{g \in \mathcal M\,:\, {\rm vol}(M, g) = 1 \,\,\,\mbox{and} \,\,\, s_g = {\rm constant}\,\}.
$$
The total scalar curvature $\mathcal S : \mathcal M_1 \to {\Bbb R}$ is defined as
$$
\mathcal S(g) = \int_M s_g\, dv_g.
$$
It is well-known that the total scalar curvature functional $\mathcal S$ restricted to $\mathcal C$ will be critical at $g$ if
and only if there is a function $f$ with $\int_M f = 0$ such that
\be
z_g = s_g'^*(f) \label{eqn2009-9-13-1}
\ee
where $z_g$ is the trace-less Ricci tensor defined as $z_g = r_g - \frac{s_g}{n}g.$
We call  (\ref{eqn2009-9-13-1}) the critical point equation (CPE hereafter).
Note that if $f=0$, then it follows from (\ref{eqn2009-9-13-1}) that $z_g = 0$ and thus, $g$ is an Einstein metric.
However, the existence of a nonzero solution is a very strong condition. The only known case with a nonzero solution
 is that of a standard sphere,  and  it has been conjectured that this is the only possible case (\cite{besse}).
 Namely, it is believed that if there  exists a non-zero function $f$ satisfying the CPE,
 then $g$ must be an Einstein metric.
 We would like to remark that  a solution $(g, f)$ to CPE is  a non-trivial example of $\mathcal A$-superharmonic condition
 since $A f = -|z_g|^2$ and $\Delta_g f = -\frac{s_g}{n-1}f$.

Unless stated otherwise, we only consider Riemannian metrics on $M$ whose scalar curvatures are positive constants.

\section{Variational Properties}

Let $(M,g)$ be a closed Riemannian $n$-manifold and  $\delta$ be the adjoint operator of
the differential $d$ with respect to the metric $g$.
Unless explicitly stated, we will use $r$ rather than $r_g$ as the Ricci tensor of the metric $g$,
and $s$ rather than $s_g$ as the scalar curvature.
The following expressions are well-known definitions and identities: for a function $\vp$ and any tensor $T$,
$$
\delta Dd\vp = - d\Delta \vp - r(d\vp, \cdot),\quad \delta d \vp = -\Delta \vp
$$
and
$$
\delta (\vp T) = \vp \delta T - T(d\vp, \cdot).
$$

\begin{lem}\label{lem2009-9-16-6}
Let $\mathcal A = s_g'\circ s_g'^*$ and assume the scalar curvature $s_g =s$ is constant. Then for any function $\vp$,
$$
\mathcal A(\vp) = (n-1)\Delta^2 \vp + 2s \Delta \vp - \la Dd\vp, r\ra + \vp |r|^2.
$$
\end{lem}
\begin{proof}
It follows directly from (\ref{eqn2009-9-12-3})  that
$$
s_g'^*(\vp) = Dd\vp - (\Delta \vp)g - \vp r
$$
and thus
$$
A(\vp) = s_g'\circ s_g'^*(\vp) = s_g'(Dd\vp - (\Delta \vp)g - \vp r).
$$
By (\ref{eqn2009-9-13-2}), we have
$$
s_g'(Dd\vp) = -\delta (r(d\vp, \cdot)) - \la Dd\vp, r\ra.
$$
Similarly, since $\delta g = 0$ and $\delta r = - \frac{1}{2} ds = 0$,  we also obtain the following from (\ref{eqn2009-9-13-2})
$$
s_g'((\Delta \vp)g) = (1-n) \Delta^2 \vp  - s \Delta \vp
$$
and
$$
s_g'(\vp r) = - s \Delta \vp + \delta(-r(d \vp, \cdot)) - \vp |r|^2.
$$
Combining these two expressions, we obtain
$$
\mathcal A(\vp) = (n-1)\Delta^2 \vp + 2s \Delta \vp - \la Dd\vp, r\ra + \vp |r|^2.
$$
\end{proof}

Note that $\mathcal A$ is a fourth-order  linear partial differential operator.
The following theorem shows that $\mathcal A$ is elliptic and self-adjoint.
We say that a fourth-order differential operator is elliptic if the symbol is injective.

\begin{thm}\label{thm2011-9-6-2}
The operator ${\mathcal A}$ is a self-adjoint, fourth-order elliptic
linear operator.
Furthermore, if $\ker s_g'^*=0$, then for any $\psi\in C^{\infty}(M)$
there exists a unique function $u\in C^{\infty}(M)$ such that
$\psi={\mathcal A}(u)$.
\end{thm}
\begin{proof}
We first show that $s_g'^*$ has injective symbol.
Recall that for any $p\in M$ and any cotangent vector $t\in T^*_pM$,
there is a linear
map $\sigma_t(s_g'^*): T_p C^{\infty}(M) \to T_p C^{\infty}(S^2M)$
called the symbol of $D=s_g'^*$, and the symbol of $D$ is called
injective if $\sigma_t(D)$ is injective for all non-zero $t$. Note
that for $t\in T^*M$, $\psi\in C^{\infty}(M)$,
$$\sigma_t(s_g'^*)\cdot \psi =(-g(t,t)g+t\otimes t )\psi, $$
which is clearly injective for $n>1$.
Thus $s_g'^*$ is an operator of order $2$ with injective symbol. By
Lemma 4.4 of \cite{b-e},
${\mathcal A}=s_g'\circ s_g'^*$ is an elliptic operator of order $4$.

To see the self-adjoint property of ${\mathcal A}$, first of all, note that for any two functions $\vp, \psi$
\be
\psi\la Dd\vp, r\ra = {\rm div}(\psi r(d\vp, \cdot)) - r(d\vp, d\psi)\label{eqn2009-9-15-10}
\ee
and so
$$
\int_M \psi\la Dd\vp, r\ra = - \int_M r(d\vp, d\psi).
$$
Integrating $\vp{\mathcal A}\psi$, we obtain
$$
\int_M \psi{\mathcal A}(\vp)= \int_M (n-1)\tr \vp \tr \psi -2s_g \langle d\vp,d \psi\rangle+r(d\vp,d \psi)+|r|^2\vp \psi.
$$
Thus,
$$
\int_M \vp{\mathcal A}  \psi = \int_M  \psi {\mathcal A} \vp.
$$

Secondly, we show that ${\mathcal A}$ is surjective.
Since $s_g'$ is surjective, for any non-trivial $\psi\in
C^{\infty}(M)$, there is $\xi\in C^{\infty}(S^2M)$ such that
$s_g'(\xi)=\psi$.
From the fact that $s_g$ is constant and the proof of Theorem 5.2 in
\cite{b-e}, $C^{\infty}(S^2M)={\rm{im}}\, s_g'^* \oplus \ker
s_g'$.
Thus, $\xi=\xi_1+\xi_2$ for $\xi_1\in {\rm{im}}\, s_g'^*$ and
$\xi_2\in \ker s_g'$. Therefore, for $\xi_1 =s_g'^*(u)$, we have
${\mathcal A}(u)=\psi$.

Finally uniqueness comes from the assumption that $\ker s_g'^*=0$
since $\ker {\mathcal A}=\ker s_g'^*$;
clearly $\ker s_g'^*\subset \ker s_g'\circ s_g'^*$, and $s_g\circ
s_g'^*(u)=0$ implies $$0=( u, s_g\circ s_g'^*(u))_{L^2}=(
s_g'^*(u),s_g'^*(u))_{L^2},$$ where $(f,g)_{L^2}=\int_M fg\,dv_g$, and
so $s_g'^*(u)=0$. \end{proof}

Given a smooth compact $n$-dimensional Riemannian manifold $(M, g)$, we let $H^2(M) = W^{2,2}(M)$ be the
Sobolev space defined as the completion of the space of smooth functions on $M$ with respect to the norm
$$
\Vert \vp\Vert^2_{H^2(M)} = \int_M |Dd\vp|^2 \, dv_g + \int_M |\nabla \vp|^2\, dv_g + \int_M \vp^2 \, dv_g.
$$
To investigate the properties of operator $\mathcal A$ from the perspective of the calculus of variations,
we define $E(\vp)$ for any function $\vp \in H^2(M)$ as
\be
E(\vp) = \frac{1}{2}\int_M \left[(n-1)(\Delta \vp)^2 -2s |d\vp|^2 + r(d\vp, d\vp) + \vp^2|r|^2\right].\label{eqn2009-11-30-3}
\ee
According to (\ref{eqn2009-9-15-10}) we have $\vp\la Dd\vp, r\ra = {\rm div}(\vp r(d\vp, \cdot)) - r(d\vp, d\vp)$  and thus
$$
\int_M \vp\la Dd\vp, r\ra = - \int_M r(d\vp, d\vp).
$$
Consequently, the Euler-Lagrange equation for the functional $E$ is exactly
$$
\mathcal A(\vp) = (n-1)\Delta^2 \vp + 2s \Delta \vp - \la Dd\vp, r\ra + \vp |r|^2.
$$
Note that if $\vp =$ constant and $\mathcal A(\vp) = 0$, then $\vp = 0$ if the Ricci curvature
$r$ does not identically vanish. Furthermore,
\be
E(\vp) = \frac{1}{2} \int_M \vp \mathcal A(\vp) = \frac{1}{2} \int_M |s_g'^* \vp|^2 \ge 0\label{eqn2009-9-17-1}
\ee
for any function $\vp$.

A simple direct observation is as follows

\begin{lem}\label{lem2009-10-2-4}
$\ker \,s_g'^* = 0$ if and only if $\ker \, \mathcal A = 0$.
\end{lem}
\begin{proof}
The proof follows from the fact that
$$
\int_M (s_g'^* \vp)^2 = \int_M \vp \mathcal A(\vp)
$$
for any function $\vp$.
In fact, assume that $\ker  \, \mathcal A = 0$ and let $s_g'^* u = 0$. Then $u$ realizes the infimum of $E(\vp)$
among all smooth functions $C^\infty(M)$.
That is, $u$ is a critical point for $E$ and thus, $\mathcal A (u) = 0$.
\end{proof}

\begin{exam}
{\rm Let $M$ be a round $n$-sphere $S^n$ with a standard round metric.
Also, let $\varphi$ be the first eigenfunction for the Laplacian so that
$$
\Delta \vp = - n \vp,\quad \int_{S^n} |d\vp|^2 = n \int_{S^n} \vp^2.
$$
Since $r =  (n-1) g$, it is easy to see that $E(\vp) = 0$. Thus the first eigenfunction
$\vp$ realizes the infimum of the functional $E$ and so
$$
\mathcal A (\vp) = 0\quad \mbox{and}\quad  \ker s_g'^* \ne 0.
$$
On the other hand, consider $M = S^2 \times S^3$ with the product metric.
It can be seen that the scalar curvature $s$,
the first non-zero eigenvalue $\lambda_1$ of the Laplacian $\Delta$, and the square norm of the Ricci tensor $|r|^2$ are given as
$$
s = 8,\quad \lambda_1 = 2,\quad |r|^2 = 14.
$$
Let $\vp$ be the first eigenfunction of the Laplacian on the first component $S^2$. Then we have
$$
\Delta \vp = - 2 \vp,\quad \int_M |d\vp|^2 = 2 \int_M \vp^2
$$
and thus, $E(\vp) = 0.$
Therefore, for $S^2\times S^3$, we obtain
$$
\mathcal A (\vp) = 0\quad  \mbox{and}\quad  \ker s_g'^* \ne 0.
$$
For higher dimensional case, let $M = S^n \times S^{n+1}$ with the standard product metric. Then,
\be
s = 2n^2,\quad |r|^2 = n(2n^2 - n+1)\label{eqn2009-11-30-1}
\ee
and the first non-trivial eigenvalue is given as
$$
\lambda_1(M) = \lambda_1(S^n) = n.
$$
Let $\varphi$ be the first eigenfunction corresponding to $\lambda_1(M)$ so that
\be
\Delta \vp = - n \vp,\quad r(d\vp, d\vp) = (n-1)|d\vp|^2.\label{eqn2009-11-30-2}
\ee
Substituting (\ref{eqn2009-11-30-1}) and (\ref{eqn2009-11-30-2}) into (\ref{eqn2009-11-30-3}), we obtain
$ E(\vp) = 0$. Therefore, we have $\mathcal A(\vp) = 0$ and thus,  $\ker s_g'^* \ne 0$.}
\end{exam}

Recall that $H^2(M) = W^{2,2}(M)$ is the Sobolev space consisting of
functions that are $L^2$ up to the second (weak) derivative. Let
$$
\mathcal W = \left\{\vp\in H^2(M)\,:\, \int_M \vp^2 = 1\,\right\}
$$
and define
$$
\nu = \inf \left\{\int_M \vp \mathcal A(\vp)\,:\, \vp \in \mathcal W\,  \right\}.
$$
Note that  $\nu \ge 0$, and $\ker \mathcal A \ne 0$ implies $\nu = 0$ by (\ref{eqn2009-9-17-1}).
The converse is also true.

\begin{thm}\label{thm2009-10-2-3}
Suppose that $\nu = 0$. Then $\ker  \mathcal A \ne 0$.
\end{thm}
\begin{proof}
Since $\nu = 0$, there exists a sequence $(\vp_k)$ of functions in $H^2(M)$ with $\int_M \vp_k^2 = 1$ such that
$$
E(\vp_k) \to 0\quad\mbox{as}\quad k\to \infty.
$$
We now claim that $(\vp_k)$ is bounded in $H^2(M)$. On the contrary, suppose  that the sequence $(\vp_k)$ is unbounded in $H^2(M)$.
Defining $\widetilde \vp_k$ as
$$
\widetilde \vp_k = \frac{\vp_k}{\Vert \vp\Vert_{H^2(M)}},
$$
where ${\Vert \vp\Vert_{H^2(M)}}$ denotes the Sobolev norm in $H^2(M)$,
we have
$$
{\Vert \vp\Vert_{H^2(M)}} =1\quad \mbox{and}\quad
\int_M \widetilde\vp_k^2 \to 0\,\,\mbox{as}\,\,\, k\to \infty.
$$
Furthermore, $E(\widetilde  \vp_k) \to 0$ as  $k\to \infty$.
Thus the rescaled sequence $(\widetilde \vp_k)$ is bounded in $H^2(M)$ and
so there exists a function $\widetilde \vp_\infty \in H^2(M)$
such that $\widetilde \vp_k \to \widetilde \vp_\infty$ in $H^2(M)$.
However since $\Vert \widetilde \vp_k \Vert_{L^2} \to 0$,
the limit function $\widetilde \vp = 0$, which is contradictory to the
fact that $\Vert d \widetilde \vp \Vert_{L^2} \ne 0$ or $\Vert Dd \widetilde\vp \Vert_{L^2} \ne 0$.

Therefore, $(\vp_k)$ is bounded,  and so  $\vp_k$ weakly converges to  a function $\vp$ in $H^2(M)$.
Applying the Rellich-Kondrakov embedding theorem $H^2(M) \subset H^1(M) \subset L^2$,
$\vp_k$ strongly converges  to $\vp$ in $L^2$, and thus, there exists a subsequence,
say $(\vp_k)$, that converges almost everywhere.
Consequently, we have
$$
E(\vp) \le \liminf_{k\to \infty} E(\vp_k) = 0.
$$
Hence since $E(\vp)= 0$ and $\int_M \vp^2 = 1$, $\vp$ is a non-constant function and $\mathcal A(\vp) = 0$.
\end{proof}

\begin{cor}\label{cor2011-7-20-1}
$\nu = 0$ if and only if $\ker \, s_g'^* \ne 0$, or if $\ker \mathcal A \ne 0$.
\end{cor}

Now we consider a special operator stemming from $\mathcal A$
that also plays a very important role in the kernel space of $s_g'^*$.
For a function $\vp$, define $P\vp$ as
$$
P\vp =  (n-1)\Delta^2 \vp + 2s_g \Delta \vp - \la Dd\vp, r_g\ra
$$
and define
$$
\mu = \inf_{\vp \in H^2(M), \vp \ne 0} \frac{\int \vp P\vp}{\int \vp^2}.
$$
Note that $\mu \le 0$ since $P\vp = 0$ when $\vp$ is a nonzero constant. Furthermore, it is easy to see
that if $\mu = 0$, then either $(M, g)$ is Ricci-flat or $\ker \mathcal A = 0$.
In fact, if $u \in \ker \mathcal A$ and $r \ne 0$, then
$$
\int_M u Pu = - \int_M u^2 |r|^2 \le 0.
$$
Since $\mu = 0$ and $r \ne 0$, $u$ must be zero because $\int_M u^2 |r|^2 = 0$.
The following theorem shows that if $\ker \mathcal A \ne 0$, then
$\mu$ must be strictly negative.

\begin{thm}\label{thm2010-2-1-1}
Assume that $\ker \mathcal A \ne 0$ and $s = s_g$ is constant. Then,
$$
-\max_M |r_g|^2 \le \mu \le - \frac{s_g^2}{n}.
$$
\end{thm}
\begin{proof}
Let $u \in \ker \mathcal A$ be a non-constant function
and $r$ be the Ricci tensor of the metric $g$. Since $\frac{s^2}{n} \le |r|^2$, we have
$$
\mu \int u^2 \le \int uPu = - \int u^2|r|^2 \le -\frac{s^2}{n}\int u^2.
$$
Thus
$$
\mu \le  - \frac{s^2}{n}.
$$
On the other hand, it follows from Lemma~\ref{lem2009-9-16-6},   (\ref{eqn2009-9-15-10}) and
the Bochner-Weitzenb\"ock formula that
\bea
\int_M (s_g'^* \vp)^2 &=&
\int_M \vp \mathcal A(\vp) \\
&=&
\int_M \left\{n (\Delta \vp)^2 - 2s |d\vp|^2  -|Dd\vp|^2 + \vp^2 |r|^2 \right\}.
\eea
Thus,
\bea
\int_M \left\{n (\Delta \vp)^2 - 2s |d\vp|^2  -|Dd\vp|^2\right\}
&\ge& - \int_M \vp^2|r|^2\\
&\ge& -\left(\max_M |r|^2\right) \int_M \vp^2.
\eea
Therefore, since
$$
\int_M \vp P\vp \ge -\left(\max_M |r|^2\right) \int_M \vp^2
$$
for any function $\vp$, we conclude that
$$
\mu \ge -\max_M |r|^2.
$$
\end{proof}

In view of Theorem~\ref{thm2010-2-1-1}, the invariant $\mu$ may designate the criteria for
how  close $g$ is to an Einstein metric.
In fact, when $(M, g)$ is Einsteinian, it follows from Theorem~\ref{thm2010-2-1-1} that
$$
\mu = - \frac{s^2}{n}
$$
if  $\ker \mathcal A \ne 0$.

In view of the operators $\mathcal A$ and $P$, for any real number $\a$,
we introduce an elliptic-fourth order partial differential equation
${\mathcal A}_\a$ defined as
$$
{\mathcal A}_\a (\vp ) = (n-1)\Delta^2 \vp + 2s_g \Delta \vp - \langle Dd\vp, r_g\rangle + (1-\a)\vp |r_g|^2,
$$
where $r_g$ is the Ricci tensor and $s_g$ is the scalar curvature which is assumed to be a positive constant.
Note that ${\mathcal A}_0 = \mathcal A$ and ${\mathcal A}_1 = P$.

\begin{thm}
Assume that $\ker \mathcal A = 0$ and $s = s_g$ is constant. Then there exists a positive real number
$\a_0 >0$ such that
$\ker {\mathcal A}_\a = 0$ for all $\a, \,0 \le \a \le \a_0$.
\end{thm}
\begin{proof}
For $0 < \a <1$, let $u \in \ker \mathcal A_\a$ be a non-trivial function. Then
$$
\mathcal A(u) = \a u |r_g|^2 \le \left(\max_M |r|^2\right)\a u
$$
and so $\nu \le \left(\max_M |r|^2\right)\a$.
Since $\ker \mathcal A = 0$, Corollary~\ref{cor2011-7-20-1} states that $\nu >0$. Hence,
$$
0 < \frac{\nu}{\max_M |r|^2} \le \a.
$$
\end{proof}

\section{Case of $\nu >0$}

In this section, we consider the case in which $\nu$ is positive, or equivalently, $\ker \mathcal A = 0$.
We will investigate some conditions and properties for $\nu$ to be positive,
and derive lower bounds on $\nu$.

\begin{lem}
Assume $\nu >0$. We then obtain
$$
\inf_{\vp\in \mathcal W, \vp \ne 1} \frac{E(\vp)}{\Vert \vp\Vert_{H^2(M)}} > 0.
$$
Here ${\Vert \vp\Vert_{H^2(M)}}$ denotes the Sobolev norm in $H^2(M)$.
\end{lem}
\begin{proof}
Suppose that $\displaystyle{\inf_{\vp\in \mathcal W, \vp \ne 1} \frac{E(\vp)}{\Vert \vp\Vert_{H^2(M)}} = 0}$.
Then, there exists a sequence
$(\vp_k)$ in $\mathcal W$ such that
$$
\frac{E(\vp_k)}{\Vert \vp\Vert_{H^2(M)}} \to 0 \quad \mbox{as}\quad k \to \infty.
$$
Since $\nu >0$, we have $\displaystyle{{\Vert \vp\Vert_{H^2(M)}} \to \infty}$ as $k \to \infty$.
Defining $\widetilde \vp_k$ as
$$
\widetilde \vp_k = \frac{\vp_k}{\Vert \vp\Vert_{H^2(M)}},
$$
we can obtain a contradiction, as in the proof of Theorem~\ref{thm2009-10-2-3}.
\end{proof}

\begin{thm}\label{thm2009-9-15-1}
Let $(M, g)$ be a compact Riemannian $n$-manifold with a positive constant scalar curvature $s$.
If $\ker \mathcal A = 0$, then $\nu>0$ is contained in the spectrum of $\mathcal A$.
\end{thm}
\begin{proof}
Recall
$$
\mathcal W = \left\{ \vp \in H^2(M)\,:\, \int_M \vp^2 = 1\,\right\}.
$$
By Theorem~\ref{thm2009-10-2-3}, we may assume that
$$
a:= \inf_{\vp\in \mathcal W, \, \vp \ne 1} \frac{E(\vp)}{\Vert \vp\Vert_{H^2(M)}} > 0.
$$
Then, for any function $\vp \in \mathcal W$, we have
$$
E(\vp) \ge a {\Vert \vp\Vert_{H^2(M)}}
$$
and thus,
$$
E(\vp) \to \infty \quad \mbox{as}\quad \Vert \vp\Vert_{H^2(M)} \to \infty.
$$
In other words, the functional $E$ is coercive on $\mathcal W$ .

On the other hand, let $(\vp_k)$ be a sequence in $H^2(M)$ such that $\vp_k \to \vp$ weakly in $H^2(M)$.
Then, according to the Rellich-Kondrakov theorem,
$\vp_k \to \vp$ strongly in $L^2(M)$ and thus, a subsequence $(\vp_k)$ converges almost everywhere.
This shows that the subspace $\mathcal W$ is weakly closed in $H^2(M)$. Furthermore, since
$M$ is compact, the subsequence $(\vp_k)$ uniformly converges to $\vp$, and we obtain
$$
E(\varphi) \le \liminf_{k\to \infty}E(\vp_k).
$$
It is a well-known fact that the functional $E$ is bounded below and attains its minimum in $H^2(M)$ (\cite{stru}).
Letting
$$
 E(u) = \min\{E(\vp): \vp\in \mathcal W\},
$$
it is easy to see from the variational principle that
$$
\mathcal A (u) = \nu u.
$$
\end{proof}

The properties of operator $\mathcal A$ and the lower bound on $\nu$ are deeply related to
 the first nonzero eigenvalue of the Laplacian.
Let $\lambda$ be the first nonzero eigenvalue for the Laplacian operator $\Delta$, which is characterized by
$$
\lambda = \inf \left\{\frac{\int_M |\na \vp|^2}{\int_M \vp^2}\,\,:\,\, \int_M \vp = 0\,\right\}.
$$
It is well-known that $\lambda$ is positive. It follows from the characterization of the first nonzero eigenvalue
that for any function $\vp$ with $\int_M \vp = 0$,
\be
\int \vp^2 \le \frac{1}{\lambda} \int |d\vp|^2.\label{eqn1000}
\ee

\begin{lem}\label{lem2009-8-30-2}
Let $(M^n, g)$ be a compact Riemannian $n$-manifold.
Then, for any function $\varphi \in C^\infty(M)$,
\be
\int_M |d\vp|^2 \le \frac{1}{\lambda}\int_M (\Delta \vp)^2 \le \frac{n}{\lambda}\int_M |Dd \vp|^2,\label{eqn2009-9-9-2}
\ee
where $\lambda$ is the first nonzero eigenvalue of the Laplacian.
\end{lem}
\begin{proof}
Applying the Cauchy-Schwarz inequality to $\vp$ and
$\bar \vp := \vp - \int_M \vp$, we can obtain
$$
\int_M |d\vp|^2 \le \frac{1}{\lambda}\int_M (\Delta \vp)^2.
$$
The second inequality follows from the fact that $(\Delta \vp)^2 \le n |Dd\vp|^2$.
\end{proof}
Furthermore, for a function $\vp$ with $\int_M \vp = 0$, we have
$$
\int_M \vp^2 \le \frac{1}{\lambda^2}\int_M (\Delta \vp)^2\quad
 \mbox{and}\quad \int_M \vp^2 \le \frac{n}{\lambda^2}\int_M |Dd\vp|^2.
$$

A direct observation from the definition of $\mathcal A$ is
the following. This theorem shows that if the first nonzero eigenvalue for the Laplacian is large compared to
the sum of the scalar curvature and the norm of the Ricci tensor, then $\nu$ is positive.

\begin{thm}
Let $(M^n, g)$ be a compact Riemannian $n$-manifold with a positive constant scalar curvature $s$. If
$\displaystyle{(n-1) \lambda \ge 2s + \max_M |r|}$, then $\nu \ge \frac{s^2}{n}$ and thus $\ker \mathcal A = 0$, or
equivalently, $\ker  s_g'^*  = 0.$
\end{thm}
\begin{proof}
Note that $|r|^2 \ge \frac{s^2}{n}$. It follows from Lemma~\ref{lem2009-8-30-2} that
$$
\int_M |d\vp|^2 \le \frac{1}{\lambda}\int_M (\Delta \vp)^2
$$
for any function $\vp$. Thus for any function $\vp \in \mathcal W$,
\bea
E(\vp)
&=& \frac{1}{2}\int_M (n-1)(\Delta \vp)^2 - 2s |d\vp|^2 + r(d\vp, d\vp) +  |r|^2 \vp^2\\
&\ge& \frac{1}{2}\left\{(n-1)\lambda - (2s+\max |r|)\right\}\int_M |d\vp|^2 + \frac{s^2}{2n}\int_M \varphi^2.
\eea
Hence, $\nu \ge \frac{s^2}{n}$.
\end{proof}

\begin{remark}
{\rm Assume $\nu >0$ for a compact Riemannian $n$-manifold $(M, g)$ with a positive constant scalar curvature. Then it follows from
Theorem~\ref{thm2009-9-15-1} that
$$
\mathcal A (u) = \nu u
$$
for some function $u \in \mathcal W$. In particular, we have
$$
\int_M u |r|^2 = \nu \int_M u.
$$
Since $\ker s_g'^* = 0$, it is easy to see that the operator $\mathcal A$ is bijective (\cite{y-c-h}).
Thus there exists a unique function $\varphi \in C^\infty(M)$ such that $\mathcal A (\vp) = u|r|^2$. Therefore
\bea
\frac{s^2}{n} &\le&
 \int_M u^2 |r|^2  = \int_M u \mathcal A \vp \\
 &=& \int_M \vp \mathcal A u = \nu \int_M \vp u\\
 &\le& \nu \Vert \vp\Vert_{L^2}
 \eea
 On the other hand, by the Cauchy-Schwarz inequality
 \bea
  \nu \Vert \vp\Vert_{L^2}^2
 &\le& \int_M \vp \mathcal A \vp = \int_M \vp u|r|^2\\
 &\le&
 \left(\int_M \vp^2 |r|^2 \right)^{\frac{1}{2}}\left(\int_M u^2|r|^2\right)^{\frac{1}{2}}\\
  &\le& \left(\max_M |r|\right)\Vert \vp \Vert_{L^2}\sqrt{\nu \Vert \vp\Vert_{L^2}}.
  \eea
Therefore, we have $\displaystyle{\nu \Vert \vp\Vert_{L^2}  \le \max_{M}|r|^2}$ and so
  $$
  \frac{s^2}{n} \le \nu \Vert \vp\Vert_{L^2}   \le \max_{M} |r|^2,
  $$
where $\vp$ is a function satisfying $\mathcal A (\vp) = u |r|^2$.}
\end{remark}

\begin{thm}
Let $M = S^m\times S^m$ $(m \ge 2)$ with the standard product metric. Then
$$
\nu = m = \frac{\dim(M)}{2}.
$$
\end{thm}
\begin{proof}
First, we will examine the case $m=2$ since key ingredients of the proof are contained in this setting.
The cases of  $m \ge 3$ will then be briefly explained.
For $M^4 = S^2 \times S^2$ with the standard product metric $g$, we obviously have
$s = |r|^2 = 4, \lambda = 2$, and $r = g$.
Thus, $\la Dd\vp, r\ra = \Delta \vp$ for any function $\varphi$, and so
$$
\mathcal A (\varphi) = 3\Delta^2 \vp + 7\Delta \vp + 4 \vp.
$$
Let $u$ be the first eigenfunction of $S^2$ so that
$\Delta u = -2 u, 2\int_M u^2 = \int_M |du|^2$ and $r(du, du) = |du|^2$. Therefore,
\bea
\int_M u \mathcal A (u) &=& \int_M 3(\Delta u)^2 - 7 |du|^2 + 4 u^2\\
&=& 2\int_M u^2
\eea
Hence $\nu \le 2$. To show the converse inequality $\nu \ge 2$,
it is sufficient to prove that for any $C^\infty$ function $\varphi$
$$
F(\vp): = \int_M \left[3(\Delta \vp)^2 - 7 |d\vp|^2 + 2 \vp^2 \right] \ge 0.
$$
First, note that
$$
F(\vp) = \int_M (\Delta \vp + 2\vp)(3\Delta \vp + \vp).
$$
It follows from Lemma~\ref{lem2009-8-30-2} that
\bea
2 \int_M |d\vp|^2 \le \int_M (\Delta \vp)^2.
\eea
Thus, from the monotonicity of eigenvalues, it follows that,
for any function $\vp$ that vanishes on the smooth boundary $\partial D$ of
a domain $D \subset M$, we have
\be
2 \int_D |d\vp|^2 \le \int_D (\Delta \vp)^2.\label{eqn2009-12-1-4}
\ee
Assume for a moment that $0$ is a regular value of $\vp$. Let $D_1$ be a region on $M$ such that
$$
\Delta \vp + 2 \vp \le 0\quad \mbox{and}\quad \Delta \vp + \vp \ge 0,
$$
and $D_2$ be a region such that
$$
\Delta \vp + 2 \vp \ge 0\quad \mbox{and}\quad \Delta \vp + \vp \le 0.
$$
Note that $\vp \le 0$ on region $D_1$, and
$\vp \ge 0$ on  region $D_2$. Thus, $\partial D_1 = \partial D_2$.
On region $D_1$,  we have
\be
0 < -\frac{1}{3}\vp \le \Delta \vp \le -2\vp.\label{eqn2009-12-1-5}
\ee
Multiplying (\ref{eqn2009-12-1-5}) by $\vp$ and integrating over $D_1$, we obtain
$$
-2 \int_{D_1} \vp^2 \le \int_{D_1} \vp \Delta \vp \le -\frac{1}{3}\int_{D_1} \vp^2.
$$
Since  $\varphi = 0$ on $\partial D_1$, we get
\be
-2 \int_{D_1} \vp^2 \le - \int_{D_1} |d\vp|^2  \le -\frac{1}{3}\int_{D_1} \vp^2.\label{eqn2009-12-1-6}
\ee
Similarly, on region $D_2$, we have
\be
-2 \int_{D_2} \vp^2 \le \int_{D_2} \vp \Delta \vp \le -\frac{1}{3}\int_{D_2} \vp^2.\label{eqn2009-12-1-7}
\ee
Let $D = D_1 \cup D_2$. It follows from (\ref{eqn2009-12-1-6}) and (\ref{eqn2009-12-1-7}) that
\be
\frac{1}{3}\int_{D}  \vp^2 \le \int_{D} |d\vp|^2  \le 2 \int_{D} \vp^2 .\label{eqn2009-12-1-6000}
\ee
Note that on $M - D$, we have
\be
(\Delta \vp + 2\vp)(3\Delta \vp + \vp) \ge 0.\label{eqn2009-12-1-8}
\ee
Furthermore, since  the function $\vp$ vanishes  on the boundary $\partial D$ of $D$,
we can apply integration by parts and Green's identity. Thus, it follows from
(\ref{eqn2009-12-1-4}), (\ref{eqn2009-12-1-6000}), and (\ref{eqn2009-12-1-8}) that
\bea
F(\vp)
&=&
\int_D (\Delta \vp + 2\vp)(3\Delta \vp + \vp) + \int_{M- D} (\Delta \vp + 2\vp)(3\Delta \vp + \vp)\\
&=&
3 \int_D \left[(\Delta \vp)^2 - 2|d\vp|^2 \right] + \int_D \left(2\vp^2 -|d\vp|^2\right) \\
&&\quad + \int_{M-D}(\Delta \vp + 2\vp)(3\Delta \vp + \vp)\\
&\ge& 0.
\eea
Now, assume that $0$ is a critical value of $\vp$. By Sard's theorem, for any positive real number $\e>0$, there
exists a real number $a, -\e < a< 0$ such that $a$ is a regular value $\vp$.
Let $D_{1,a}$ be a region such that
\be
\Delta \vp + 2\vp \le \frac{5}{3}a\quad \mbox{and}\quad \Delta \vp + \frac{\vp}{3} \ge 0.\label{eqn2011-7-18-1}
\ee
Note that $\vp \le a <0$ on region $D_{1,a}$, and  $\vp = a$ on the boundary $\partial D_{1,a}$.
Multiplying (\ref{eqn2011-7-18-1}) by $\vp$ and integrating it over $D_{1,a}$, we obtain
\be
\frac{5}{3}a \int_{D_{1,a}} \vp - a
\int_{\partial D_{1,a}}\frac{\partial \vp}{\partial n_1} \le
\int_{D_{1,a}} \left(2\vp^2 - |d\vp|^2\right),\label{eqn2011-7-18-2}
\ee
where $n_1$ is a unit normal vector field to $\partial D_{1,a}$. Next, let $D_{2,a}$ be a region such that
\be
\Delta \vp + 2\vp \ge -\frac{5}{3}a\quad \mbox{and}\quad \Delta \vp + \frac{\vp}{3} \le 0.\label{eqn2011-7-18-3}
\ee
We may assume that $-a$ is also a regular value of $\vp$.
Note that $0< -a  \le \vp$ on  region $D_{2,a}$, and $\vp = -a$ on the boundary $\partial D_{2,a}$.
Multiplying (\ref{eqn2011-7-18-3}) by $ \vp$ and integrating it over $D_{2,a}$, we obtain
\be
a\int_{\partial D_{2,a}}\frac{\partial \vp}{\partial n_2} - \frac{5}{3}a \int_{D_{2,a}} \vp
\le  \int_{D_{2,a}} \left(2\vp^2 - |d\vp|^2\right),\label{eqn2011-7-18-4}
\ee
where $n_2$ is a unit normal vector field on $\partial D_{2,a}$.
Decomposing $M$ into three regions, we can write
\bea
F(\vp)
&=&
3 \int_{D_{1,a}} \left[(\Delta \vp)^2 - 2|d\vp|^2 \right] + \int_{D_{1,a}} \left(2\vp^2 -|d\vp|^2\right) \\
&&\quad +
3 \int_{D_{2,a}} \left[(\Delta \vp)^2 - 2|d\vp|^2 \right] + \int_{D_{2,a}} \left(2\vp^2 -|d\vp|^2\right)\\
&&\quad +
 \int_{M-(D_{1,a}\cup D_{2,a})}(\Delta \vp + 2\vp)(3\Delta \vp + \vp).
\eea
Applying inequality (\ref{eqn2009-12-1-4}) to $\varphi - a$, we have
$$
\int_{D_{1,a}} \left[(\Delta \vp)^2 - 2|d\vp|^2 \right] \ge 0
$$
and
$$
\int_{D_{2,a}} \left[(\Delta \vp)^2 - 2|d\vp|^2 \right] \ge 0.
$$
Thus, from (\ref{eqn2011-7-18-2}) and (\ref{eqn2011-7-18-4}), we obtain
\bea
F(\vp)
&\ge&  \frac{5}{3}|a| \int_{D_{1,a}\cup D_{2,a}} |\vp| - a\int_{\partial D_{1,a}}\frac{\partial \vp}{\partial n_1}
+ a\int_{\partial D_{2,a}}\frac{\partial \vp}{\partial n_2}\\
&&\quad + \int_{M-(D_{1,a}\cup D_{2,a})}(\Delta \vp + 2\vp)(3\Delta \vp + \vp).
\eea
Since $|\frac{\partial \vp}{\partial n_1}| \le |d\vp|$ and $|\frac{\partial \vp}{\partial n_2}| \le |d\vp|$,
the first three terms on the right-hand side approach to $0$ as $\e \to 0$.
Finally, let $E_{1, a}$  be a region such that
$\Delta \vp + 2\vp > \frac{5}{3}a$ and $\Delta \vp + \frac{\vp}{3} \ge 0$, and $E_{2,a}$ be a region such that
$\Delta \vp + 2\vp < - \frac{5}{3}a$ and $\Delta \vp + \frac{\vp}{3} \le 0$.
Then, we have
\bea
&&\int_{M-(D_{1,a}\cup D_{2,a})}(\Delta \vp + 2\vp)(3\Delta \vp + \vp)\\
&&\qquad
\ge \frac{5}{3}a \int_{E_{1,a}}(3\Delta \vp + \vp)
- \frac{5}{3}a \int_{E_{2,a}}(3\Delta \vp + \vp).
\eea
The right-hand side approaches to $0$ as $\e \to 0$. Hence, $F(\vp) \ge 0$.

In the general case,  $M^{2m} = S^m \times S^m$ when $m \ge 2$, it is easy to see that
$$
s = 2m(m-1),\quad |r|^2 = 2m(m-1)^2, \quad r = (m-1)g,\quad \lambda = m.
$$
Thus,
\bea
\int_M \vp \mathcal A (\vp) &=& (2m-1)\int_M \left[(\Delta \vp)^2 - m |d\vp|^2\right]\\
&&\quad -(2m^2 - 4m+1)\int_M |d\vp|^2 + 2m(m-1)^2 \int_M \vp^2.
\eea
Using the first eigenfunction $u$ of $S^m$, $\Delta u = -m u$, we can demonstrate that $\nu \le m$.
To show that $\nu \ge m$, it is sufficient to prove that
for any function $\vp$
$$
F(\vp) : = \int_M (\Delta \vp + m \vp)\left[(2m-1)\Delta \vp + (2m^2 - 4m+1)\vp\right] \ge 0.
$$
Note that
$$
m \int_M |d\vp|^2 \le \int_M (\Delta \vp)^2.
$$
An argument identical to that used in the case  $S^2 \times S^2$ shows that $F(\vp) \ge 0$, and thus, $\nu = m$.
\end{proof}

\begin{remark}
{\rm For the case of $M = S^m \times S^{m+k}$ with $k\ge 2$, the first nonzero eigenfunction of $S^m$ can be used to show that
$$
\nu \le \min\{\,(m+k)(k-1)^2, \,\, m (k+1)^2 \,\}.
$$
However, we do not know the exact lower bound on $\nu$.}
\end{remark}

\section{The First Eigenvalue of the Laplacian}

As mentioned above, the first non-zero eigenvalue $\lambda = \lambda_1(M)$ of the Laplace operator for
a Riemannian manifold $(M, g)$ is related to the operator
$\mathcal A$. For example, if $\ker \mathcal A \neq 0$ and $g$ is an Einstein metric with a positive scalar curvature,
then $\lambda  =\frac s {n-1}$ from the results obtained by  Lichnerowicz (\cite{b-g-m}) and Bourguignon (\cite{bour}).
We shall now see that, if there is a non-trivial function
on which the action of $\mathcal A$  is non-positive where the function is positive,
then the first nonzero eigenvalue of the Laplacian is bounded above and vice versa.
Recall that we assumed that the scalar curvature $s_g = s$ of a Riemannian manifold $(M, g)$ is always a positive constant.

For a function $\vp$ on a smooth manifold $M$, let us denote
$$
M^+_\vp = \{x\in M\, :\, \vp(x) >0\,\}.
$$
We say that  a Riemannian manifold $(M, g)$ satisfies the $\mathcal A$-{\it superharmonic condition} if there
exists a smooth function $\vp$ such that
\begin{itemize}
\item[(i)]  $M^+_\vp \ne \emptyset$ and $\mathcal A \vp \le 0$ on $M^+_\vp$,
\item[(ii)] $\Delta \vp = 0$ on the boundary $\partial M_\vp^+$ of $M_\vp^+$.
\end{itemize}
For example, if $M = S^n$ with the standard round metric $g_0$, and $\vp$ is the first nonzero eigenfunction of the Laplacian, i.e.,
$\Delta \vp = -n \vp$, then $\mathcal A \vp = 0$ and  $(S^n,g_0)$ satisfies the $\mathcal A$-superharmonic condition.
Furthermore, note that any eigenfunction of the Laplacian satisfies the second condition (ii).
The following lemma shows that the $\mathcal A$-superharmonic condition
 is implied by $\ker \mathcal A \ne 0$.

\begin{lem}\label{lem2011-7-31-2}
Let $(M^n, g)$ be a compact $n$-dimensional Riemannian manifold with a positive constant scalar curvature $s_g$.
If $\ker \mathcal A \ne 0$, then $(M, g)$ satisfies the $\mathcal A$-superharmonic condition.
\end{lem}
\begin{proof}
By Lemma~\ref{lem2009-10-2-4}, $\ker \mathcal A \ne 0$ is equivalent to $\ker s_g'{^*} \ne 0$.
 Let $s_g'{^*} \vp = 0$ and $\vp \ne 0$. Then,
 $$
 Dd\vp - (\Delta \vp)g - \vp r_g =0.
 $$
In particular, taking the trace yields
$$
\Delta \vp = -\frac{s_g}{n-1} \vp
$$
and so $M_\vp^+ \ne \emptyset$. On the other hand, taking the inner product with the Ricci tensor $r_g$ of $g$, we obtain
$$
\langle Dd \vp, r_g\rangle - s_g \Delta \vp - \vp |r_g|^2 = 0.
$$
Thus, it follows from the definition of $\mathcal A$ that
 \bea
 \mathcal A \vp &=& (n-1)\Delta^2 \vp + 2s_g \Delta \vp - \langle Dd\vp, r_g\rangle + \vp |r_g|^2\\
 &=&
 \frac{s_g^2}{n-1}\vp - \frac{s_g^2}{n-1}\vp\\
 &=& 0.
 \eea
 Hence, the  function $\vp$ satisfies conditions (i) and (ii) in the definition of the $\mathcal A$-superharmonic condition.
\end{proof}

\begin{thm}\label{thm2009-12-20-1-10}
Let $(M^n, g)$ be a compact $n$-dimensional Riemannian manifold with a positive constant scalar curvature $s_g$.
Suppose that $(M, g)$ satisfies the $\mathcal A$-superharmonic condition.
If $r_g \ge k$ for $0 < k \le 2\left(1 - \sqrt{1- \frac{1}{n}}\right)s_g$, then
 the first eigenvalue of the Laplacian satisfies
\be
\lambda = \lambda_1(M) \le \frac{2s_g-k + \sqrt{k^2 - 4ks_g + \frac{4s_g^2}{n}}}{2(n-1)}.\label{eqn2011-7-30-2}
\ee
\end{thm}
\begin{proof}
Let $s_g = s$ and $r_g = r$. In addition, let $\vp$ be a smooth function satisfying
$M^+_\vp \ne \emptyset$, $\mathcal A \vp \le 0$ on $M^+_\vp$, and $\Delta \vp = 0$ on the boundary $\partial M_\vp^+$.
If $\vp$ is a constant function, then $\vp$ is a positive constant since $M_\vp^+ \ne \emptyset$. However, we have
$\mathcal A \vp = \vp |r|^2 \le 0$, which is a contradiction.
Thus, we may assume that $\vp$ is a non-constant function.
By the above hypothesis, we have
\be
\int_{M^+_\vp} \vp \mathcal A \vp  \le 0.\label{eqn1-1-10}
\ee
By the definition of $\mathcal A$ and integration by parts, together with the fact that $\Delta \vp = 0$ on $\partial M_\vp^+$,
we obtain
\be
\int_{M^+_\vp} \vp {\mathcal A}\vp
&=&
\int_{M^+_\vp} (n-1)(\Delta \vp)^2 - \int_{\partial M_\vp^+} \Delta \vp \frac{\partial \vp}{\partial \nu}\nonumber\\
&&\quad + \int_{M_\vp^+} \left[2s \vp \Delta \vp + r(d\vp, d\vp) + |r|^2 \vp^2\right]\nonumber\\
&\ge&
\int_{M^+_\vp} \left[(n-1)(\Delta \vp)^2 + (2s-k) \vp \Delta \vp + \frac{s^2}{n} \vp^2 \right].\label{eqn2011-7-30-1}
\ee
Note that
\be
&&(n-1)(\Delta \vp)^2 + (2s-k) \vp \Delta \vp + \frac{s^2}{n} \vp^2\label{eqn10-10}\\
&&\qquad = \left((n-1)\Delta \vp + \a \vp\right)(\Delta \vp + \beta \vp),\nonumber
\ee
where
$$
\alpha = \frac{2s-k+\sqrt{k^2 - 4ks + \frac{4s^2}{n}}}{2},\,\,\,\,
\beta = \frac{2s-k -\sqrt{k^2- 4ks+\frac{4s^2}{n}}}{2(n-1)}.
$$
Also note that
$$
k = 2s\left(1- \sqrt{1-\frac{1}{n}}\right)\quad \Rightarrow \quad
k^2 - 4ks + \frac{4s^2}{n} = 0.
$$
{\bf Contention:}\,\, If
\be
\lambda = \lambda_1(M) > \frac{2s - k + \sqrt{k^2 - 4k s + \frac{4s^2}{n}}}{2(n-1)} = \frac{\a}{n-1},\label{eqn100-100}
\ee
then any subset $\O$ of ${M^+_\vp}$ with a $C^1$ boundary
on which $(n-1)\Delta \vp + \a \vp \ge 0$ and $\Delta \vp + \b \vp \le 0$ has a measure of zero.

Suppose that a subset $\O$ of $M_\vp^+$  contains an open $n$-ball.
Note that since $\Delta \vp = \vp = 0$ on $\partial \O$, we can apply the Dirichlet principle on
the first nonzero eigenvalue of the Laplacian. By monotonicity, we have
$$
\lambda \le \lambda_1(\O).
$$
Since $(n-1)\Delta \vp + \a \vp \ge 0$ and $\vp > 0$ on $\O$, we have
$$
\vp \Delta \vp \ge - \frac{\a}{n-1} \vp^2.
$$
Integrating this over $\O$, we obtain
$$
\int_{\O} |d\vp|^2 \le \frac{\a}{n-1}\int_{\O} \vp^2 \le \frac{\a}{n-1}\cdot \frac{1}{\lambda_1(\O)}\int_{\O}|d\vp|^2.
$$
Thus,
$$
1 \le \frac{\a}{n-1}\cdot \frac{1}{\lambda_1(\O)}
$$
and so
$$
\lambda \le \lambda_1(\O) \le \frac{\a}{n-1},
$$
which contradicts  (\ref{eqn100-100}). This completes the proof of contention.

Now, suppose that $\lambda > \frac{\a}{n-1}$. Since $\a > (n-1)\b$, it follows from (\ref{eqn10-10}) and the above contention that
$$
(n-1)(\Delta \vp)^2 + (2s-k) \vp \Delta \vp + \frac{s^2}{n} \vp^2 \ge 0 \quad \mbox{a.e on}\,\,\, {M^+_\vp},
$$
which implies  that $\int_{M^+_\vp} \vp \mathcal A \vp \ge 0$. Consequently, from (\ref{eqn1-1-10}), we have
$$
\int_{M^+_\vp} \vp \mathcal A \vp = 0.
$$
Thus, on the set $M_\vp^+$,  we have $\mathcal A \vp = 0$ and
\bea
&&(n-1)(\Delta \vp)^2 + (2s-k) \vp \Delta \vp + \frac{s^2}{n} \vp^2 \\
&&\qquad =\left((n-1)\Delta \vp + \a \vp\right)(\Delta \vp + \beta \vp) \\
&&\qquad
= 0
\eea
by (\ref{eqn2011-7-30-1}). Since $\a > (n-1)\b$, either $(n-1)\Delta \vp + \a \vp = 0$ or
$\Delta \vp + \beta \vp = 0$ on the entire set $M_\vp^+$.
Therefore,  we obtain
$$
\lambda \le \lambda_1(M_\vp^+) \le \max\left\{\frac{\a}{n-1}, \beta\right\} = \frac{\a}{n-1},
$$
which contradicts the assumption $\lambda > \frac{\a}{n-1}$. Hence,
$$
\lambda = \lambda_1(M) \le \frac{\a}{n-1}.
$$
\end{proof}

\begin{remark}
{\rm From the work of  Lichnerowicz (\cite{b-g-m}), if a compact Riemannian manifold $(M^n, g)$ satisfies $r_g \ge k >0$, then
$$
\lambda = \lambda_1(M) \ge \frac{n}{n-1}k.
$$
On the other hand, if $M = S^n$ with the standard round metric,
$$
n-1 \le 2\left(1 - \sqrt{1- \frac{1}{n}}\right)s_g.
$$
Thus, taking $k= n-1$, the right-hand side in inequality (\ref{eqn2011-7-30-2})
becomes
$$
\frac{2s_g-k + \sqrt{k^2 - 4ks_g + \frac{4s_g^2}{n}}}{2(n-1)} = n
$$
and so the result in Theorem~\ref{thm2009-12-20-1-10} is optimal.
Also, note that if $k = 2s\left(1 - \sqrt{1-\frac{1}{n}}\right)$, then
$k^2 - 4ks_g + \frac{4s_g^2}{n}= 0$ and thus, the following fact holds:
if $(M, g)$ satisfies the $\mathcal A$-superharmonic condition and
$r_g \ge 2s\left(1 - \sqrt{1-\frac{1}{n}}\right)$, then
$$
\lambda_1(M) \le \frac{s}{2\sqrt{n(n-1)}}.
$$ }
\end{remark}

Without the condition on $k$ in Theorem~\ref{thm2009-12-20-1-10} and with only the nonnegativity of the Ricci curvature,
we obtain the following slightly weaker result. In fact, this is the $k = 0$ case in Theorem~\ref{thm2009-12-20-1-10}.

\begin{cor}\label{thm2009-12-20-1}
Let $(M^n, g)$ be a compact $n$-dimensional Riemannian manifold such that
the  Ricci curvature is nonnegative and the scalar curvature $s_g$ is a positive constant.
In addition, suppose that $(M, g)$ satisfies the $\mathcal A$-superharmonic condition.
Then, the first nonzero eigenvalue of the Laplacian satisfies
$$
\lambda = \lambda_1(M) \le \frac{s_g}{n-1}\left(1 + \frac{1}{\sqrt{n}}\right).
$$
\end{cor}
\begin{proof}
Let $s_g = s$, $\mathcal A \vp \le 0$ on $M_\vp^+ \ne \emptyset$, and $\Delta \vp = 0$ on $\partial M_\vp^+$
 for a function $\vp$. Then, we have
\be
0 &\ge& \int_{M^+_\vp} \vp {\mathcal A}\vp\label{eqn2009-12-24-1}\\
&=&
\int_{M^+_\vp} (n-1)(\Delta \vp)^2 + 2s \vp \Delta \vp + r(d\vp, d\vp) + |r|^2 \vp^2\nonumber\\
&\ge&
\int_{M^+_\vp} \left[(n-1)(\Delta \vp)^2 + 2s \vp \Delta \vp + \frac{s^2}{n} \vp^2 \right] + \int_{M^+_\vp} r(d\vp, d\vp)\nonumber\\
&=:& I_1 + I_2.\nonumber
\ee
Note that
$$
(n-1)(\Delta \vp)^2 + 2s \vp \Delta \vp + \frac{s^2}{n} \vp^2
= \left((n-1)\Delta \vp + a \vp\right)(\Delta \vp + b \vp),
$$
where
$$
a = \left(1 + \frac{1}{\sqrt{n}}\right)s,\,\,\,\, b = \frac{1}{n-1}\left(1 - \frac{1}{\sqrt{n}}\right) s.
$$
As in the proof of Theorem~\ref{thm2009-12-20-1-10}, we can show that if
$$
\lambda = \lambda_1(M) > \frac{s}{n-1}\left(1 + \frac{1}{\sqrt{n}}\right) = \frac{a}{n-1},
$$
then any subset $\O$ with a $C^1$ boundary of ${M^+_\vp}$ on which $(n-1)\Delta \vp + \a \vp \ge 0$
and $\Delta \vp + \b \vp \le 0$ has a measure of zero.
Since $r_g \ge 0$, this implies that
$$
(n-1)(\Delta \vp)^2 + 2s \vp \Delta \vp + \frac{s^2}{n} \vp^2 = 0 \quad \mbox{a.e on}\,\,\, {M^+_\vp}.
$$
Since $a > (n-1)b$, either $(n-1)\Delta \vp + a \vp = 0$ or $\Delta \vp + b \vp = 0$ on the entire set $M_\vp^+$. This
contradicts the assumption $\lambda_1(M) > \frac{a}{n-1}$ from the monotonicity of the eigenvalues of the Laplacian.
\end{proof}

\begin{remark}
Let $(M^n, g)$ be a compact $n$-dimensional Riemannian manifold such that
$r_g \ge k$ for $0 \le k \le 2\left(1 - \sqrt{1- \frac{1}{n}}\right)s_g$, where the scalar curvature
$s_g$ is a positive constant.
In addirion, suppose that there exists a function $\vp$ such that ${M^-_\vp} = \{x\in M \, :\, \vp(x) < 0 \,\}\ne \emptyset$
and $\mathcal A \vp \ge 0$ on ${M^-_\vp}$.
Then, the same argument used  in the proof of Theorem~\ref{thm2009-12-20-1-10} shows that
the first nonzero eigenvalue satisfies
$$
\lambda = \lambda_1(M) \le \frac{2s_g-k + \sqrt{k^2 - 4ks_g + \frac{4s_g^2}{n}}}{2(n-1)}.
$$
In particular, if $k = 0$, then
$$
\lambda = \lambda_1(M) \le \frac{s}{n-1}\left(1 + \frac{1}{\sqrt{n}}\right).
$$
\end{remark}

Finally, we consider the relationship of $\nu$ with the first nonzero eigenvalue of the Laplace operator.
In the case of $\nu >0$, it follows from
Theorem~\ref{thm2009-9-15-1} that a minimizer $u$ for the functional $E$
satisfies ${\mathcal A}u = \nu u$. In particular, since $\ker s_g'^* = 0$ when $\nu >0$, we cannot, in general, expect that
$\frac{s_g}{n-1}$ is contained in the spectrum of the Laplace operator.

\begin{thm}\label{thm5} Let $(M,g)$ be a compact $n$-dimensional Riemannian manifold such that $r_g\geq k \geq 0$
and assume that $\nu >\frac{s_g^2}{n}$, where $s_g$ is a positive constant. In addition, suppose that
 $M_{u}^+\neq \emptyset$ for a function $u$ satisfying $\mathcal A u = \nu u$.
 Then, the first nonzero eigenvalue of the Laplacian satisfies
$$
\lambda_1(M) \le \frac {2s_g-k+\sqrt{k^2-4ks_g +\frac {4s_g^2}n +4(n-1)\nu}}{2(n-1)},
$$
unless $(M,g)$ is Einsteinian.
\end{thm}
\begin{proof}
We shall denote $s_g$ as $s$ and $r_g$ as $r$.
From $\int_M u {\mathcal A}u = \nu \int_M u^2$,
\begin{eqnarray*}
0&=&\int_M u {\mathcal A}u -\nu u^2 \\
&=& \int_M (n-1)(\tr u)^2 + 2 s u\tr u + r(du,du)+(|r|^2-\nu) u^2\\
&\geq& \int_M (n-1)(\tr u)^2+(2s-k)u \tr u + \left(\frac {s^2}n-\nu \right) u^2
\end{eqnarray*}
We may factor the integrand as follows:
\bea
&&(n-1)(\tr u)^2+(2s-k)u \tr u + \left(\frac {s^2}n-\nu \right) u^2\\
&&\qquad
= ((n-1)\tr u + {\alpha}u) \left(\tr u +\frac {\beta}{n-1}u \right),
\eea
where
\begin{eqnarray*}
\alpha&=&\frac 12\left({2s-k+\sqrt{k^2-4ks+\frac {4s^2}n+4(n-1)\nu}}\right),\\
\beta &=&\frac 12 \left({2s-k-\sqrt{k^2-4ks+\frac {4s^2}n+4(n-1)\nu}}\right).
\end{eqnarray*}
Note that if $\nu >\frac{s^2}{n}$, $k^2-4ks+\frac {4s^2}n+4(n-1)\nu>0$ for any $k\geq 0$.

The remainder of the proof is similar to that of Theorem~\ref{thm2009-12-20-1-10}.
Hence, if $g$ is not an Einstein memtric and $\lambda > \frac {\alpha}{n-1} $, then
$$
0\geq \int_M u {\mathcal A}u -\nu u^2> 0,
$$
which is a contradiction.
\end{proof}

\begin{cor} Let $(M,g)$ be a compact $n$-dimensional Riemannian manifold of positive Ricci curvature.
In addition, assume that  $M_{u}^+\neq \emptyset$ for a function $u$ satisfying
$\mathcal A u = \nu u$. Then, the first nonzero eigenvalue of the Laplacian satisfies
$$
\lambda_1 \le \frac {s_g+\sqrt{\frac {s_g^2}n +(n-1)\nu}}{(n-1)}.
$$
\end{cor}


\begin{thm}
Let $(M,g)$ be a compact $n$-dimensional Riemannian manifold such that $r_g\geq k$ with
$0 \le k \leq 2s_g\left(1-\sqrt{1-\frac 1n-(n-1)\frac{\nu}{s_g^2}}\right)$.
Suppose that $0< \nu \leq \frac{s_g^2}{n}$.
In addition, assume that $M_{u}^+\neq \emptyset$ for a function $u$ satisfying $\mathcal A u = \nu u$.
Then, the first nonzero eigenvalue of the Laplacian satisfies
$$
\lambda_1 \le \frac {2s_g-k+\sqrt{k^2-4ks_g +\frac {4s_g^2}n +4(n-1)\nu}}{2(n-1)},
$$
unless $(M,g)$ is Einstein.
\end{thm}
\begin{proof} Note that if $\displaystyle{\nu\leq \frac{s_g^2}{n}}$ and
$0 \leq k\leq 2s_g \left(1-\sqrt{1-\frac 1n- (n-1)\frac{\nu}{s_g^2}}\right)$, then
$$
k^2-4ks_g +\frac {4s_g^2}n +4(n-1)\nu\geq 0.
$$
The remainder of the proof proceeds in the same manner as Theorem~\ref{thm5}.
\end{proof}








\end{document}